\documentclass[12pt,a4paper]{amsart}
\usepackage[centertags]{amsmath}
\usepackage{amsfonts}
\usepackage{amssymb}
\usepackage{amsthm}
\usepackage{newlfont}
\newcommand{\Z}{\Bbb Z}
\newcommand{\N}{\Bbb N}
\newcommand{\Q}{\Bbb Q}

\def\NL{\hfill\break}
\def\ni{\noindent}
\newcommand{\set}[1]{\left\{#1\right\}}
\newcommand{\parth}[1]{\left(#1\right)}
\newtheorem{theo}{Theorem}[section]
\newtheorem{df}[theo]{Definition}
\newtheorem{cor}[theo]{Corollary}
\newtheorem{lem}[theo]{Lemma}
\newtheorem{ex}[theo]{Example}
\newtheorem{prop}[theo]{Proposition}

\begin{document}

\title{Ore Extensions of Extended Symmetric and Reversible Rings}

\maketitle

\begin{center}
{\bf Mohamed Louzari and L'moufadal Ben Yakoub}
\end{center}
{\small
\begin{center}
Dept.\ of Mathematics, Abdelmalek Essaadi University
\end{center}
\begin{center}
Faculty of sciences, B.P. 2121 Tetouan, Morocco
\end{center}
\begin{center}
benyakoub@hotmail.com, mlouzari@yahoo.com
\end{center}}

\begin{abstract}
Let $\sigma$ be an endomorphism and $\delta$ an $\sigma$-derivation of a ring $R$. In this paper, we show that if $R$ is $(\sigma,\delta)$-skew Armendariz and $a\sigma(b)=0$ implies $ab=0$ for $a,b\in R$. Then $R$ is symmetric (respectively,  reversible) if and only if $R$ is $\sigma$-symmetric (respectively,  $\sigma$-reversible) if and only if $R[x;\sigma,\delta]$ is symmetric (respectively,  reversible). Moreover, we study on the relationship between the Baerness, quasi-Baerness and p.q.-Baerness of a ring $R$ and these of the Ore extension $R[x;\sigma,\delta]$. As a consequence we obtain a partial generalization of \cite{hong/2000}.
\end{abstract}

\footnote[0]{

\ni 2000 Mathematics Subject Classification. 16U80, 16S36
\NL Key words and phrases. Armendariz rings, symmetric rings, reversible rings, Ore extensions}

\section*{Introduction}

Throughout this paper, $R$ denotes an associative ring with unity. $\sigma$ is a ring endomorphism, and $\delta$ an $\sigma$-derivation of $R$, that is, $\delta$ is an additive map such that $\delta(ab)=\sigma(a)\delta(b)+\delta(a)b$ for all $a,b\in R$. We denote $R[x;\sigma,\delta]$ the Ore extension whose elements are polynomials over $R$, the addition is defined as usual and the multiplication subject to the relation $xa=\sigma(a)x+\delta(a)$ for all $a\in R$. A ring $R$ is called {\it symmetric} if $abc=0$ implies $acb=0$ for all $a,b,c\in R$. A ring $R$ is called {\it
reversible} if $ab=0$ implies $ba=0$ for all $a,b\in R$. Reduced rings (i.e., rings with no nonzero nilpotent elements) are symmetric by Anderson and Camillo \cite[Theorem~1.3]{anderson}. Commutative rings are clearly symmetric, symmetric rings are clearly reversible. Polynomial rings over reversible rings need not to be reversible, and polynomial rings over symmetric rings need not to be symmetric (see \cite{kim} and \cite{wang}). From \cite{baser/2007}, a ring $R$ is called {\it right $($respectively,  left$)$ $\sigma$-reversible} if whenever $ab=$ for $a,b\in R$, $b\sigma(a)=0$ (respectively,  $\sigma(b)a=0$). Also, by \cite{kwak}, a ring $R$ is called {\it right $($respectively,  left$)$ $\sigma$-symmetric} if whenever $abc=0$ for $a, b, c \in R$, $ac\sigma(b)=0$ (respectively,  $\sigma(b)ac=0$). Rege and Chhawchharia \cite{chhawchharia}, called a ring $R$ an {\it Armendariz} if whenever polynomials $f=\sum_{i=0}^{n}a_ix^i,\;g=\sum_{j=0}^{m}b_jx^j\in R[x]$ satisfy $fg=0$, then $a_ib_j=0$ for each $i,j$. The term Armendariz was introduced by Rege and Chhawchharia \cite{chhawchharia}. This nomenclature was used by them since it was Armendariz \cite[Lemma 1]{armendariz1}, who initially showed that a reduced ring always satisfies this condition. According to Krempa \cite{krempa}, an endomorphism $\sigma$ of a ring $R$ is called to be {\it rigid} if $a\sigma(a)=0$ implies $a=0$ for all $a\in R$. A ring $R$ is called $\sigma$-{\it rigid} if there exists a rigid endomorphism $\sigma$ of $R$. Note that any rigid endomorphism of a ring $R$ is a monomorphism and $\sigma$-rigid rings are reduced by Hong et al. \cite{hong/2000}. Also, by \cite[Theorem 2.8(1)]{kwak}, a ring $R$ is $\sigma$-rigid if and only if $R$ is semiprime right $\sigma$-symmetric and $\sigma$ is a monomorphisme, so right $\sigma$-symmetric ($\sigma$-reversible) rings are a generalization of  $\sigma$-rigid rings. Properties of $\sigma$-rigid rings have been studied in \cite{hashemi/quasi, hong/2000, hong/2003, krempa}. In \cite{hong/2003}, Hong et al. defined a ring $R$ with an endomorphism $\sigma$ to be $\sigma$-{\it skew Armendariz} if whenever polynomials $f=\sum_{i=0}^{n}a_ix^i,\;g=\sum_{j=0}^{m}b_jx^j\in R[x;\sigma]$ satisfy $fg=0$ then $a_i\sigma^i(b_j)=0$ for each $i,j$. According to Hong et al. \cite{hong/2006}. A ring $R$ is said to be $\sigma$-{\it Armendariz}, if whenever polynomials $f=\sum_{i=0}^{n}a_ix^i,\;g=\sum_{j=0}^{m}b_jx^j\in
R[x;\sigma]$ satisfy $fg=0$ then $a_ib_j=0$ for each $i,j$. From Hashemi and Moussavi \cite{hashemi/skew}, a ring $R$ is called a $(\sigma,\delta)$-{\it skew Armendariz } ring if for $p=\sum_{i=0}^{n}a_ix^i$ and $q=\sum_{j=0}^{m}b_jx^j$ in $R[x;\sigma,\delta]$, $pq=0$ implies $a_ix^ib_jx^j=0$ for each $i,j$. By Hashemi and Moussavi \cite{hashemi/quasi}, a ring $R$ is $\sigma$-{\it compatible} if for each $a,b\in R$, $a\sigma(b)=0$ if and only if $ab=0$. Moreover, $R$ is said to be $\delta$-{\it compatible} if for each $a,b\in R$, $ab=0$ implies $a\delta(b)=0$. If $R$ is both $\sigma$-compatible and $\delta$-compatible, we say that $R$ is $(\sigma,\delta)$-{\it compatible}. A ring $R$ is $\sigma$-rigid if and only if $R$ is $(\sigma,\delta)$-compatible and reduced \cite[Lemma~2.2]{hashemi/quasi}.

\par In this paper, we study the transfert of symmetry ($\sigma$-symmetry) and reversibility ($\sigma$-reversibility) from $R$ to $R[x;\sigma,\delta]$. We show, that if $R$ is $(\sigma,\delta)$-skew Armendariz and $a\sigma(b)=0$ implies $ab=0$ for $a,b\in R$. Then $R$ is symmetric (respectively  reversible) if and only if $R$ is $\sigma$-symmetric (respectively  $\sigma$-reversible) if and only if if $R[x;\sigma,\delta]$ is symmetric (respectively,  reversible). As a consequence we obtain a generalization of \cite[Theorem 3.6]{hong/2006}, \cite[Proposition 3.4]{huh}, \cite[Proposition 2.4]{kim}, \cite[Corollary 2.11]{baser/2007} and \cite[Theorem 2.10]{kwak}. A connection between reversibility (respectively,  symmetry) and $\sigma$-reversibility (respectively,  $\sigma$-symmetry) of a ring is given. Moreover, we study on the relationship between the Baerness, quasi-Baerness and p.q.-Baerness of a ring $R$ and these of the Ore extension $R[x;\sigma,\delta]$. As a consequence we obtain a partial generalization of \cite[Theorem 11, Corollaries 12 and 15]{hong/2000}.

\section{Preliminaries and Examples}
\smallskip

We begin with the following definition.

\begin{df}\label{def}Let $\sigma$ be an endomorphism of a ring $R$. We say that $R$
satisfies the condition $(\mathcal{C_{\sigma}})$ if whenever $a\sigma(b)=0$ with $a,b\in R$, then $ab=0$.
\end{df}

In the Ore extension $R[x;\sigma\,\delta]$, we have for $n\geq 0$, $x^na=\sum_{i=0}^nf_i^n(a)x^i$ where $f_i^n\in End(R,+)$ will denote the map which is the sum of all possible words in $\sigma,\delta$ built with $i$ letters $\sigma$ and $n-i$ letters $\delta$. (In particular, $f_n^n=\sigma^n, f_0^n=\delta^n $), \cite[Lemma~4.1]{lam}.

\begin{lem}\label{lem1}Let $R$ be an $(\sigma,\delta)$-skew Armendariz ring satisfying the condition $(\mathcal{C_{\sigma}})$ and $n\geq 1$ a natural number. If $ab=0$ then $\sigma^n(a)b=\delta^n(a)b=0$ for all $a,b\in R$.
\end{lem}

\begin{proof}It suffices to show the result for $n=1$. Let $f=\sigma(a)x+\delta(a)$ and
$g=b$ such that $ab=0$. Then $fg=\sigma(a)xb+\delta(a)b=\sigma(ab)x+\sigma(a)\delta(b)+
\delta(a)b=\sigma(ab)x+\delta(ab)=0$. By Proposition \ref{prop3.2}, we have $\sigma(a)b=\delta(a)b=0$.
\end{proof}

\begin{lem}\label{lem3}Let $R$ be an $(\sigma,\delta)$-skew Armendariz reversible ring satisfying the condition $(\mathcal{C_{\sigma}})$. If $ab=0$, then $af_i^j(b)=0$ for all $i,j\;\; (i\leq j)$.
\end{lem}

\begin{proof}Let $a,b\in R$ such that $ab=0$, since $R$ is reversible we have $ba=0$ and
so $\sigma^n(b)a=\delta^n(b)a=0$, by Lemma \ref{lem1}. Also, we have  $a\sigma^n(b)=a\delta^n
(b)=0$ for all $n\in\N$. Thus $af_i^j(b)=0$ for all $i,j$.
\end{proof}

In the next, we show some connections between $(\sigma,\delta)$-skew Armendariz rings, $\sigma$-rigid rings and rings with the condition $(\mathcal{C}_{\sigma})$.

\begin{ex}\label{ex2.1} Let $\Z$ be the ring of integers and $\Z_4$ be
the ring of integers modulo 4. Consider the ring
$$R=\set{\begin{pmatrix}
  a & \overline{b} \\
  0 & a
\end{pmatrix}| a\in\Z\;,\overline{b}\in \Z_4}.$$  Let $\sigma\colon R\rightarrow R$ be an
endomorphism defined by $\sigma\left(\begin{pmatrix}
  a & \overline{b} \\
  0 & a
\end{pmatrix}\right)=\begin{pmatrix}
  a & -\overline{b} \\
  0 & a
\end{pmatrix}$.\NL Take the ideal $I=\set{\begin{pmatrix}
  a & \overline{0} \\
  0 &a
\end{pmatrix}| a\in 4\Z}$ of $R$. Consider the factor ring $$R/I\cong\set{\begin{pmatrix}
  \overline{a} & \overline{b} \\
  0 &\overline{a}
\end{pmatrix}| \overline{a},\overline{b}\in 4\Z}.$$ \NL$(i)$ $R/I$ is not
$\overline{\sigma}$-skew Armendariz. In fact, $\left(\begin{pmatrix}
 \overline{2} & \overline{0} \\
  0 & \overline{2}
\end{pmatrix}+\begin{pmatrix}
 \overline{2} & \overline{1} \\
  0 & \overline{2}
\end{pmatrix}x\right)^2=0\in(R/I)[x;\overline{\sigma}]$, but $\begin{pmatrix}
 \overline{2} & \overline{1} \\
  0 & \overline{2}
\end{pmatrix}\overline{\sigma}\begin{pmatrix}
 \overline{2} & \overline{0} \\
  0 & \overline{2}
\end{pmatrix}\neq 0$.\NL$(ii)$ $R/I$ satisfies the condition $(\mathcal{C}_{\overline{\sigma}})$. Let $A=\begin{pmatrix}
 \overline{a} & \overline{b} \\
  0 & \overline{a}
\end{pmatrix}\;,B=\begin{pmatrix}
 \overline{a'} & \overline{b'} \\
  0 & \overline{a'}
\end{pmatrix}\in R/I$. If $A\overline{\sigma}(B)=0$ then $\overline{aa'}=0$ and
$\overline{ab'}=\overline{ba'}=0$, so that $AB=0$.
\end{ex}

\begin{ex}\label{ex2.2}Consider a ring of polynomials over $\Z_2$,
$R=\Z_2[x]$. Let $\sigma\colon R\rightarrow R$ be an endomorphism
defined by $\sigma(f(x))=f(0)$. Then:
\NL $(i)$ $R$ does not satisfy the condition $(\mathcal{C}_{\sigma})$. Let $f=\overline{1}+x$, $g=x\in R$, we have
$fg=(\overline{1}+x)x\neq 0$, however $f\sigma(g)=(\overline{1}+x)\sigma(x)=0$.
\NL $(ii)$ $R$ is $\sigma$-skew Armendariz \cite[Example~5]{hong/2003}.
\end{ex}

\begin{ex}\label{ex2.3}Consider the ring $$R=\set{\begin{pmatrix}
  a & t \\
  0 & a
\end{pmatrix}| a\in \Z\;,t\in \Q},$$ where $\Z$ and $\Q$ are the
set of all integers and all rational numbers, respectively. The ring
$R$ is commutative, let $\sigma\colon R\rightarrow R$ be an
automorphism defined by $\sigma\left(\begin{pmatrix}
  a & t \\
  0 & a
\end{pmatrix}\right)=\begin{pmatrix}
  a & t/2 \\
  0 & a
\end{pmatrix}$.\NL
$(i)$ $R$ is not $\sigma$-rigid. $\begin{pmatrix}
  0 & t \\
  0 & 0
\end{pmatrix}\sigma\left(\begin{pmatrix}
  0 & t \\
  0 & 0
\end{pmatrix}\right)=0$, but $\begin{pmatrix}
  0 & t \\
  0 & 0
\end{pmatrix}\neq 0$, if $t\neq 0$.\NL $(ii)$ $R$ satisfies the condition $(\mathcal{C_{\sigma}})$. Let $\left(%
\begin{array}{cc}
  a & t \\
  0 & a \\
\end{array}%
\right)$ and $\left(%
\begin{array}{cc}
  b & x \\
  0 & b \\
\end{array}%
\right)\in R$ such that $$\left(%
\begin{array}{cc}
  a & t \\
  0 & a \\
\end{array}%
\right)\sigma\left(\left(%
\begin{array}{cc}
  b & x \\
  0 & b \\
\end{array}%
\right)\right)=0,$$ hence $ab=0=ax/2+tb$, so $a=0$ or $b=0$. In each
case,
$ax+tb=0$, hence $\left(%
\begin{array}{cc}
  a & t \\
  0 & a \\
\end{array}%
\right)\left(%
\begin{array}{cc}
  b & x \\
  0 & b \\
\end{array}%
\right)=0$. We have also the same for the converse.
\NL $(iii)$ $R$ is $\sigma$-skew Armendariz \cite[Example 1]{hong/2000}.
\end{ex}

Examples \ref{ex2.1} and \ref{ex2.2} shows that the $(\sigma,\delta)$-skew Armendariz property of a ring and the condition $(\mathcal{C_{\sigma}})$ are independent of each other. From \cite[Lemma 2.2]{hashemi/quasi}, \cite[Lemma 2.5]{louzari1} and Example \ref{ex2.3}. The class of $\sigma$-rigid rings is strictly included in the class of $(\sigma,\delta)$-skew Armendariz rings satisfying the condition $(\mathcal{C_{\sigma}})$.

\section{Reversibility and Symmetry of Ore Extensions}

\smallskip

From Isfahani and Moussavi \cite{isfahani}, a ring $R$ is called {\it skew Armendariz}, if for $f=\sum_{i=0}^{n}a_ix^i,\;g=\sum_{j=0}^{m}b_jx^j\in R[x;\sigma,\delta]$, $fg=0$ implies $a_0b_j=0$ for all $j$. Obviously, $(\sigma,\delta)$-skew Armendariz rings satisfying the condition $(\mathcal{C_{\sigma}})$ are skew Armendariz. But the converse is not true by Example \ref{ex2.4}.

\begin{prop}\label{prop3.2}Let $R$ be a ring, $\sigma$ an endomorphism, and $\delta$ an $\sigma$-derivation of $R$. If one of the following conditions is satisfied.\NL
$(i)$ $R$ is $(\sigma,\delta)$-skew Armendariz and satisfies the condition $(\mathcal{C_{\sigma}})$;
\NL$(ii)$ $R$ is skew Armendariz and $R[x;\sigma,\delta]$ is symmetric;\NL Then, for $f=\sum_{i=0}^{n}a_ix^i$ and $g=\sum_{j=0}^{m}b_jx^j$ $\in R[x;\sigma,\delta]$, $fg=0$ implies $a_ib_j=0$ for all $i,j$.
\end{prop}

\begin{proof}Let $f=\sum_{i=0}^na_ix^i$, $g=\sum_{j=0}^mb_jx^j\in R[x;\sigma,\delta]$ such that $fg=0$.
\NL(i) Since $R$ is $(\sigma,\delta)$-skew Armendariz, then $a_ix^ib_jx^j=0$ for all $i,j$. Or
$a_ix^ib_jx^j=a_i\sum_{\ell=0}^if_{\ell}^i(b_j)x^{j+\ell}=a_i\sigma^i(b_j)x^{i+j}+p(x)=0$
where $p(x)$ is a polynomial of degree strictly less than $i+j$. Thus $a_i\sigma^i(b_j)=0$, and by the condition $(\mathcal{C_{\sigma}})$ we have $a_ib_j=0$ for all $i,j$.
\NL(ii) We have $(a_0+a_1x+\cdots+a_nx^n)(b_0+b_1x+\cdots+b_mx^m)=a_0(b_0+b_1x+\cdots+b_mx^m)+ (a_1x+\cdots+a_nx^n)(b_0+b_1x+\cdots+b_mx^m)=0$, or $a_0b_j=0$ for all $j$, because $R$ is skew Armendariz. So that
$(a_1x+\cdots+a_nx^n)(b_0+b_1x+\cdots+b_mx^m)=(a_1+\cdots+a_nx^{n-1})x (b_0+b_1x+\cdots+b_mx^m)=0$. Since
$R[x;\sigma,\delta]$ is symmetric, then we have $(a_1x+\cdots+a_nx^n)(b_0+b_1x+\cdots+b_mx^m)x=0$, hence $a_1b_j=0$
for all $j$. Continuing this process, we have $a_ib_j=0$ for all $i,j$.
\end{proof}

Proposition \ref{prop3.2} shows that, $\sigma$-skew Armendariz rings satisfying the condition $(\mathcal{C_{\sigma}})$ are $\sigma$-Armendariz.

\begin{ex}\label{ex2.4}Consider the ring $R=\Z_2[x]$. Let $\sigma\colon R\rightarrow R$ be
an endomorphism defined by $\sigma(f(x))=f(0)$. Then
\NL $(i)$ $R[y;\sigma]$ is not reversible $($so is not
symmetric$)$: Let $f=ay$ , $g=b\in R[y;\sigma]$ with
$a=\overline{1}+x$ and $b=x$, then $fg=ayb=a\sigma(b)y=0$. But,
$gf=bay=x(\overline{1}+x)y\neq 0$.
\NL $(ii)$ $R$ does not satisfy the condition $(\mathcal{C_{\sigma}})$ $($see Example \ref{ex2.2}$)$.
\NL $(iii)$ $R$ is skew Armendariz: Since $R$ is $\sigma$-skew Armendariz, by \cite[Example~5]{hong/2003}.
Then \cite[Theorem 2.2]{isfahani} implies that $R$ is skew Armendariz.
\NL $(iv)$ $R$ is not $\sigma$-Armendariz by \cite[Example 1.9]{hong/2006}.
\end{ex}

By Example \ref{ex2.4}, the conditions ``$R[x;\sigma,\delta]$ is symmetric'' and ``the condition $(\mathcal{C_{\sigma}})$'' in Proposition \ref{prop3.2} are not superfluous.

\begin{lem}\label{lem4}Let $R$ be an $(\sigma,\delta)$-skew Armendariz ring satisfying the condition $(\mathcal{C_{\sigma}})$. Then, for $f=\sum_{i=0}^na_ix^i, g=\sum_{j=0}^mb_jx^j$ and $h=\sum_{k=0}^pc_kx^k\in R[x;\sigma,\delta]$, if $fgh=0$ then $a_ib_jc_k=0$ for all $i,j,k$.
\end{lem}

\begin{proof}Let $f=\sum_{i=0}^na_ix^i$, $g=\sum_{j=0}^mb_jx^j$ and $h=\sum_{k=0}^pc_kx^k$ $\in R[x;\sigma,\delta]$. Note that if $fg=$ then $a_ig=0$ for all $i$. Suppose that $fgh=0$ then $a_i(gh)=0$ for all $i$, so $(a_ig)h=0$ for all $i$. Then by Proposition \ref{prop3.2}, we have $a_ib_jc_k=0$ for all $i,j,k$.
\end{proof}

\begin{theo}\label{oretheo}Let $R$ be an $(\sigma,\delta)$-skew Armendariz ring satisfying the condition $(\mathcal{C_{\sigma}})$. Then
\par$(1)$ $R$ is reversible if and only if $R[x;\sigma,\delta]$ is reversible;
\par$(2)$ $R$ is symmetric if and only if $R[x;\sigma,\delta]$ is symmetric.
\end{theo}

\begin{proof}Note that any subring of symmetric (respectively,  reversible) ring is again symmetric (respectively,  reversible). Conversely, let $f=\sum_{i=0}^na_ix^i,g=\sum_{j=0}^mb_jx^j$ and $h=\sum_{k=0}^pc_kx^k\in R[x;\sigma,\delta]$.
\NL(1) If $fg=0$ then $a_ib_j=0$ for all $i,j$ (by Proposition \ref{prop3.2}). Since $R$ is reversible, we have $b_ja_i=0$ for all $i,j$, Lemma \ref{lem3} implies that $b_jf_k^{\ell}(a_i)=0$ for all $i,j$. Thus $gf=\sum_{i=0}^n\sum_{j=0}^m\sum_{k=0}^jb_jf_k^j(a_i)x^{i+k}=0$.
\NL(2) If $fgh=0$ then $a_ib_jc_k=0$ for all $i,j,k$, by Lemma \ref{lem4}. Since $R$ is symmetric, $a_ic_kb_j=0$ for all $i,j,k$. Also, $R$ is reversible then by Lemma \ref{lem3}, we have $a_if_{s}^i(c_kf_{t}^k(b_j))=0$ for all $i,j,k,s,t$ with $(s\leq i,\,t\leq k)$. Thus $fhg=0$.
\end{proof}

According to Hong et al. \cite[Theorem 1.8]{hong/2006}, if $R$ is $\sigma$-Armendariz then $R$ is $\sigma$-skew Armendariz. But the converse is note true by \cite[Example 1.9]{hong/2006}.

\begin{theo}\label{sig/arm sk/arm}$R$ is $\sigma$-Armendariz if and only if $R$ is $\sigma$-skew Armendariz and satisfies the condition  $(\mathcal{C}_{\sigma})$.
\end{theo}

\begin{proof}$(\Rightarrow)$. By \cite[Proposition 1.3(ii) and Theorem 1.8]{hong/2006}. $(\Leftarrow)$. Let $f=a_0+a_1x+\cdots+a_nx^n$ and $g=b_0+b_1x+\cdots+b_mx^m$ $\in R[x;\sigma]$, such that $fg=0$, since $R$ is $\sigma$-skew Armendariz then $a_i\sigma^i(b_j)=0$ for all $ i,j$. Since $R$ satisfies the condition $(\mathcal{C}_{\sigma})$ then $a_ib_j=0$ for all $i,j$. Therefore $R$ is $\sigma$-Armendariz.
\end{proof}

\smallskip

We clearly obtain the following corollaries of Theorem \ref{oretheo}.

\begin{cor}[{\cite[Theorem 3.6]{hong/2006}}]\label{cor2}Let $R$ be an $\sigma$-Armendariz ring. Then
\par$(1)$ $R$ is reversible if and only if $R[x;\sigma]$ is reversible.
\par$(2)$ $R$ is symmetric if and only if $R[x;\sigma]$ is symmetric.
\end{cor}

\begin{cor}[{\cite[Proposition 3.4]{huh} and \cite[Proposition 2.4]{kim}}]Let $R$ be an Armendariz ring. Then
\par$(1)$ $R$ is reversible if and only if $R[x]$ is reversible.
\par$(2)$ $R$ is symmetric if and only if $R[x]$ is symmetric.
\end{cor}

\section{$\sigma$-Reversibility and $\sigma$-Symmetry of Ore Extensions}

\smallskip

In the next Lemma we give a relationship between $\sigma$-reversibility (respectively,  $\sigma$-symmetry) and reversibility (respectively,  symmetry).

\smallskip

\begin{lem}\label{lem5}Let $R$ be a ring and $\sigma$ an endomorphism of $R$. If $R$ satisfies the condition $(\mathcal{C_{\sigma}})$. Then
\par $(1)$ $R$ is reversible if and only if $R$ is $\sigma$-reversible;
\par $(2)$ $R$ is symmetric if and only if $R$ is $\sigma$-symmetric.
\end{lem}

\begin{proof}$(1)$ Let $a,b\in R$, $ab=0$ implies $b\sigma(a)=0$, with the condition $(\mathcal{C_{\sigma}})$, we have $ba=0$. So $R$ is reversible. Conversely, let $a,b\in R$, suppose that $ab=0$. If $b\sigma(a)\neq 0$ (i.e., $R$ is not right $\sigma$-reversible), the reversibility of $R$ gives $\sigma(a)b\neq 0$. Also, $R$ satisfies $(\mathcal{C_{\sigma}})$, then $\sigma(ab)\neq 0$. Contradiction. Now, if $\sigma(b)a\neq 0$ (i.e., $R$ is not left $\sigma$-reversible), the condition $(\mathcal{C_{\sigma}})$ gives $\sigma(ba)\neq 0$. Contradiction, because $R$ is reversible.\NL$(2)$ Let $a,b,c\in R$ such that $abc=0$. We have $bca=0$ (by reversibility), so $ca\sigma(b)=0$ (by right $\sigma$-reversibility), then $\sigma(b)ca=0$ (by reversibility), and so $\sigma(b)ac=0$. So we have the left $\sigma$-symmetry. With the same method, we obtain the right $\sigma$-symmetry. Conversely, if $abc=0$, by right $\sigma$-symmetry we have $ac\sigma(b)=0$, then $acb=0$ by the condition $(\mathcal{C_{\sigma}})$.
\end{proof}

\begin{ex}\label{ex2.5}Let $\Z_2$ is the ring of integers modulo $2$, take $R=\Z_2\oplus\Z_2$ with the usual addition and multiplication. $R$ is commutative, and so $R$ is symmetric $($so reversible$)$. Now, consider $\sigma\colon R\rightarrow R$ defined by $\sigma((a,b))=(b,a)$. Then, we have:
\NL$(i)$ $R$ is not $\sigma$-reversible $($so not $\sigma$-symmetric$)$: For $a=(1,0)$ and $b=(0,1)$, we have $ab=0$, but $b\sigma(a)=b^2=b\neq 0$.\NL$(ii)$ $R$ does not satisfy the condition $(\mathcal{C_{\sigma}})$: $a\sigma(a)=0$, but $a^2=a\neq 0$.
\end{ex}

By Example \ref{ex2.5}, we see that ``the condition $(\mathcal{C_{\sigma}})$'' in Lemma \ref{lem5}, is not superfluous.

\begin{theo}\label{s rev theo}Let $R$ be an $(\sigma,\delta)$-skew Armendariz ring satisfying the condition $(\mathcal{C_{\sigma}})$. The following statements are equivalent:
\par$(1)$ $R$ is reversible;
\par$(2)$ $R$ is $\sigma$-reversible;
\par$(3)$ $R$ is right $\sigma$-reversible;
\par$(4)$ $R[x;\sigma,\delta]$ is reversible.
\end{theo}

\begin{proof}$(1)\Leftrightarrow (4)$. By Theorem \ref{oretheo}.
\NL$(1)\Rightarrow (2)$ and $(2)\Rightarrow (3)$. Immediately from Lemma \ref{lem5}.
\NL$(3)\Rightarrow (1)$. Let $a,b\in R$, if $ab=0$ then $b\sigma(a)=0$ (right $\sigma$-reversibility), so $ba=0$ (condition $(\mathcal{C_{\sigma}})$).
\end{proof}

\begin{theo}\label{s sym theo}Let $R$ be an $(\sigma,\delta)$-skew Armendariz ring satisfying the condition $(\mathcal{C_{\sigma}})$. The following statements are equivalent:
\par$(1)$ $R$ is symmetric;
\par$(2)$ $R$ is $\sigma$-symmetric;
\par$(3)$ $R$ is right $\sigma$-symmetric;
\par$(4)$ $R[x;\sigma,\delta]$ is symmetric.
\end{theo}

\begin{proof}As of Theorem \ref{s rev theo}.
\end{proof}

The next corollaries are direct consequences of Theorems \ref{sig/arm sk/arm}, \ref{s rev theo} and \ref{s sym theo}.

\begin{cor}[{\cite[Corollary 2.11]{baser/2007}}]Let $R$ be an $\sigma$-Armendariz ring. The following statements are equivalent:
\par$(1)$ $R$ is reversible;
\par$(2)$ $R$ is $\sigma$-reversible;
\par$(3)$ $R$ is right $\sigma$-reversible;
\par$(4)$ $R[x;\sigma]$ is reversible.
\end{cor}

\begin{cor}[{\cite[Theorem 2.10]{kwak}}]Let $R$ be an $\sigma$-Armendariz ring. The following statements are equivalent:
\par$(1)$ $R$ is symmetric;
\par$(2)$ $R$ is $\sigma$-symmetric;
\par$(3)$ $R$ is right $\sigma$-symmetric;
\par$(4)$ $R[x;\sigma]$ is symmetric.
\end{cor}

\section{Related Topics}

In this section we turn our attention to the relationship between the Baerness, quasi-Baerness and p.q.-Baerness of a ring $R$ and these of the Ore extension $R[x;\sigma,\delta]$ in case $R$ is right $\sigma$-reversible and satisfies the condition $(\mathcal{C_{\sigma}})$. For a nonempty subset $X$ of $R$, we write $r_R(X)=\{c\in R|dc=0\;\mathrm{for\;any}\;d\in X\}$ which is called the right annihilator of $X$ in $R$.

\begin{lem}\label{lem4}If $R$ is a right $\sigma$-reversible ring with $\sigma(1)=1$. Then
\par$(1)$ $\sigma(e)=e$ and $\delta(e)=0$ for all idempotent $e\in R$;
\par$(2)$ $R$ is abelian.
\end{lem}

\begin{proof}(1) Let $e$ an idempotent of $R$. We have $e(1-e)=(1-e)e=0$ then $(1-e)\sigma(e)=e\sigma((1-e))=0$, so $\sigma(e)-e\sigma(e)=e-e\sigma(e)=0$, therefore $\sigma(e)=e$. Also, $\delta(e^2)=e\delta(e)+\delta(e)e$ implies $(1-e)\delta(e)=e\delta(e)$, so $(1-e)(1-e)\delta(e)=0$, thus $\delta(e)=e\delta(e)$. On other hand $(1-e)e\delta(e)=e\delta(e)$ implies $e\delta(e)=0$. Therefore $\delta(e)=0$. (2) Let $r\in R$ and $e$ an idempotent of $R$. We have $e(1-e)=0$ then $e(1-e)r=0$, since $R$ is right $\sigma$-reversible then $(1-e)r\sigma(e)=0=(1-e)re=0$, so $re=ere$. Since $(1-e)e=0$, we have also $er=ere$. Then $R$ is abelian.
\end{proof}

Kaplansky \cite{Kaplansky}, introduced the concept of {\it Baer rings} as rings in which the right (left) annihilator of every nonempty subset is generated by an idempotent. According to Clark \cite{clark}, a ring $R$ is called {\it quasi-Baer} if the right annihilator of each right ideal of $R$ is generated (as a right ideal) by an idempotent. It is well-known that these two concepts are left-right symmetric.

Birkenmeier, Kim and Park \cite{birk/pqBaer}, called $R$ a {\it right p.q.-Baer} (principally quasi-Baer) ring if the right annihilator of a principal right ideal of $R$ is generated by an idempotent. Similarly, left p.q.-Baer rings can be defined. $R$ is called a p.q.-Baer ring if it is both right and left p.q.-Baer. The class of p.q.-Baer rings has been extensively investigated by them \cite{birk/pqBaer}. This class includes all biregular rings and all (quasi-) Baer rings.

\begin{theo}\label{theo baer qbaer}Let $R$ be a right $\sigma$-reversible ring which satisfies the condition $(\mathcal{C_{\sigma}})$ with $\sigma(1)=1$. Then
\par$(1)$ If $R$ is a Baer ring  then $R[x;\sigma,\delta]$ is a Baer ring;
\par$(2)$ If $R$ is a quasi-Baer ring  then $R[x;\sigma,\delta]$ is a quasi-Baer ring.
\par$(3)$ If $R$ is a p.q.-Baer ring  then $R[x;\sigma,\delta]$ is a p.q.-Baer ring.
\end{theo}

\begin{proof}$(1)$. Suppose that $R$ is Baer. Let $A$ be a nonempty subset of $R[x;\sigma,\delta]$ and $A^*$ be the set of all coefficients of elements of $A$. Then $A^*$ is a nonempty subset of $R$ and so $r_R(A^*)=eR$ for some idempotent element $e\in R$. Since $e\in r_{R[x;\sigma,\delta]}(A)$ by Lemma \ref{lem4}. We have $eR[x;\sigma,\delta]\subseteq r_{R[x;\sigma,\delta]}(A)$. Now, let $0\neq q=b_0+b_1x+b_2x^2+\cdots+b_mx^m\in r_{R[x;\sigma,\delta]}(A)$. Then $Aq=0$ and hence $pq=0$ for any $p\in A$. Let $p=a_0+a_1x+a_2x^2+\cdots+a_nx^n$, then
$pq=\sum_{i=0}^n\sum_{k=0}^m\parth{a_i\sum_{j=0}^if_j^i(b_k)}x^{j+k}=0.$
So, we have the following system of equations:
$$a_n\sigma^n(b_m)=0;\eqno(0)$$
$$a_n\sigma^n(b_{m-1})+a_{n-1}\sigma^
{n-1}(b_m)+a_nf_{n-1}^n(b_m)=0;\eqno(1)$$
$$a_n\sigma^n(b_{m-2})+a_{n-1}\sigma^{n-1}(b_{m-1})+
a_nf_{n-1}^n(b_{m-1})+a_{n-2}\sigma^{n-2}(b_m)\eqno(2)$$
$$+a_{n-1}f_{n-2}^{n-1} (b_m)+a_nf_{n-2}^n(b_m)=0;$$
$$a_n\sigma^n(b_{m-3})+a_{n-1}\sigma^{n-1}(b_{m-2})+
a_nf_{n-1}^n(b_{m-2})+a_{n-2}\sigma^{n-2}(b_{m-1})\eqno(3)$$
$$+a_{n-1}f_{n-2}^{n-1}(b_{m-1})+a_nf_{n-2}^n(b_{m-1})+a_{n-3}
\sigma^{n-3}(b_m)+a_{n-2}f_{n-3}^{n-2}(b_m)$$
$$+a_{n-1}f_{n-3}^{n-1}(b_m)+a_nf_{n-3}^n(b_m)=0;$$
$$\cdots\cdots\cdots\cdots$$
$$\sum_{j+k=\ell}\;\;\sum_{i=0}^n\;
\sum_{k=0}^m\parth{a_i\sum_{j=0}^if_j^i(b_k)}=0;\eqno(\ell)$$
$$\cdots\cdots\cdots\cdots$$
$$\sum_{i=0}^na_i\delta^i(b_0)=0.\eqno(n+m)$$

From Eq. $(0)$, we have $a_nb_m=0$ then $b_m\in r_R(A^*)=eR$. From Eq. $(1)$, we have $a_n\sigma^n(b_{m-1})=0$, because $a_{n-1}\sigma^{n-1}(b_m)=a_nf_{n-1}^n(b_m)=0$. So, we have $a_nb_{m-1}=0$ then $b_{m-1}\in r_R(A^*)=eR$. From Eq. $(2)$, we have $a_n\sigma^n(b_{m-2})=0$, because $a_{n-1}\sigma^{n-1}(b_{m-1})=a_nf_{n-1}^n(b_{m-1})=a_{n-2}\sigma^{n-2}(b_m)=a_{n-1}f_{n-2}^{n-1}(b_m)=a_nf_{n-2}^n(b_m)=0$. So $b_{m-2}\in r_R(A^*)=eR$. Continuing this procedure yields $b_m,b_{m-1},b_{m-2},\cdots, b_0\in r_R(A^*)$. So, we can write $q=eb_0+eb_1x+eb_2x^2+\cdots+b_mx^m\in eR[x;\sigma,\delta]$. Therefore $eR[x;\sigma,\delta]=r_{R[x;\sigma,\delta]}(A)$. Consequently, $R[x;\sigma,\delta]$ is a Baer ring.
\par$(2)$ The proof for the case of the quasi-Baer property follows in a similar fashion; In fact, for any right ideal $A$ of $R[x;\sigma,\delta]$, take $A^*$ as the right ideal generated by all coefficients of elements of $A$.
\par$(3)$ Since $R$ is right $\sigma$-reversible then $Re$ is $(\sigma,\delta)$-stable for all idempotents $e\in R$ by Lemma \ref{lem4}. So by \cite[Proposition 3.1]{louzari2}, we have the result.
\end{proof}

By \cite[Lemma 2.2]{hashemi/quasi}, $\sigma$-rigid rings satisfy the condition $(\mathcal{C_{\sigma}})$, so Theorem \ref{theo baer qbaer} is a partial generalization of \cite[Theorem 11, Corollaries 12 and 15]{hong/2000}.

\section*{acknowledgments}The authors express their gratitude to Professor Laiachi EL Kaoutit for valuable remarks and helpful comments. The second author wishes to thank Professor Amin Kaidi of University of Almer\'ia for his generous hospitality. This work was supported by the project PCI Moroccan-Spanish A/011421/07.

\end{document}